\documentclass[12pt]{article}
\usepackage{amssymb}
\usepackage{amsfonts}
\usepackage{amsmath}
\usepackage[usenames]{color}
\usepackage{mathrsfs}
\usepackage{amsfonts}
\usepackage{amssymb,amsmath}
\usepackage{CJK}
\usepackage{cite}
\usepackage{cases}
\usepackage{amsthm}
\usepackage{verbatim}

\def\cl{\centerline}

\def\vs{\vspace*}

\def\L{\mathcal{L}}
\def\H{\mathcal{H}}

\def\Z{\mathbb{Z}}
\def\N{\mathbb{N}}

\def\C{\mathbb{C}}

\numberwithin{equation}{section}
\newtheorem{theo}{Theorem}[section]

\newtheorem{coro}[theo]{Corollary}
\newtheorem{lemm}[theo]{Lemma}

\newtheorem{prop}[theo]{Proposition}

\newtheorem{remark}[theo]{Remark}

\begin{document}
\begin{center}
{\large\bf New irreducible non-weight Virasoro modules \\ from tensor products}
\end{center}
\cl{Haibo Chen, Jianzhi Han, Yanyong Hong and Yucai Su}

\footnote{$^{*}$ Corresponding author.  Haibo Chen, School of  Statistics and Mathematics, Shanghai Lixin University of  Accounting and Finance,   Shanghai
201209, China, e-mail: hypo1025@163.com}
\footnote{
Jianzhi Han,
 School of Mathematical Sciences, Tongji University, Shanghai 200092, China,
e-mail:  jzhan@tongji.edu.cn}
\footnote{ Yanyong Hong,
School of Science, Hangzhou Normal University, Hangzhou 311300, China,
e-mail:  hongyanyong2008@yahoo.com}
\footnote{Yucai Su, School of Mathematical Sciences, Tongji University, Shanghai
200092, China,
 e-mail: ycsu@tongji.edu.cn}

\noindent{{\bf Abstract:}  The  non-weight tensor product  Virasoro  modules
$\mathcal{M}(V,\lambda_0)\otimes\bigotimes_{i=1}^m\Omega(\lambda_i,\alpha_i)$ are studied in this paper,
where $\Omega(\lambda_i,\alpha_i)$  and $\mathcal{M}(V,\lambda_0)$ (with $\lambda_0,\lambda_i,\alpha_i\in\C^*$ and $V$ being an irreducible module over certain finite dimensional Lie algebra related to  the Virasoro algebra)  are
irreducible Virasoro modules respectively defined in \cite{LZ2} and \cite{LZ}.
The necessary and sufficient condition for $\mathcal{M}(V,\lambda_0)\otimes\bigotimes_{i=1}^m\Omega(\lambda_i,\alpha_i)$ to be irreducible is obtained. Then we determine the necessary and sufficient condition for two such irreducible  modules to be isomorphic.
At last, we show that the irreducible $\L$-modules $\mathcal{M}(V,\lambda_0)\otimes\bigotimes_{i=1}^m\Omega(\lambda_i,\alpha_i)$    are new for $m>1$.}
\vs{5pt}



\section{Introduction}
The  {\em  Virasoro algebra} $\L$     is
 an infinite-dimensional complex Lie algebra
with the basis  $\{L_m,C
\mid m\in \Z\}$ and  the  Lie bracket defined as follows:
\begin{equation*}
[L_m,L_n]= (n-m)L_{m+n}+\delta_{m+n,0}\frac{m^{3}-m}{12}C\ {\rm and}\ [L_m,C]=0\ {\rm for}\ m,n\in\Z,
\end{equation*}
which is a one-dimensional central extension of the Witt algebra.
It  is well-known that  $\L$ is a very important   infinite-dimensional Lie algebra   both in mathematics and mathematical  physics (see, e.g., \cite{IK,LL,KR}).

The theory of  weight modules over $\L$ has been well developed (see, e.g., \cite{IK}).
The most important class of weight modules are the  highest weight modules, which  depend on the triangular decomposition structure of $\L$.
  In fact, any irreducible  weight module over the Virasoro algebra with  a nonzero finite-dimensional weight space is a  Harish-Chandra module (see \cite{MZ1}), i.e., all  weight subspaces are all finite-dimensional. And the classification of irreducible Harish-Chandra modules over $\L$ was already achieved  (see \cite{M,S,MP}).
Afterwards, several families of irreducible weight modules with infinite-dimensional weight spaces were also investigated (see, e.g., \cite{CM,LZ0,LLZ}).
Such a module was first constructed by taking the tensor product of some highest weight module and intermediate
series module (see \cite{Z}), whose irreducibility   was  solved completely  in \cite{CGZ}.

And  non-weight $\L$-modules have also attracted much attention in the past few years, such as Whittaker modules (see, e.g., \cite{OW,LGZ,MW,MZ}),
 $\C[L_0]$-free modules,
irreducible modules from Weyl modules and a class of non-weight modules  including highest-weight-like modules
 (see, e.g, \cite{LZ,TZ,TZ1,LZ2,GLZ,CG}). In the present paper, we shall study non-weight $\L$-modules. More precisely, we are going to construct a family of new irreducible $\L$-modules from the tensor products of a finite number of  $\L$-modules  $\Omega(\lambda_i,\alpha_i)=\C[\partial_i]$ (see \cite{LZ2}) and $\L$-module $\mathcal{M}(V,\lambda_0)=V\otimes\C[\partial_0]$ (see \cite{LZ}).
It is worthwhile to  point out that  tensor product modules  over the Virasoro algebra as defined in \cite{H,CGZ,TZ,TZ1,LGW,Z} are related to the locally finite modules $\mathrm{Ind}(M)$ (see \cite{MZ}). While tensor product modules  whose  tensor factors are locally non-nilpotent are also important. In this present paper we shall investigate a class of such kind of modules.  To be explicit, we are going to determine the irreducibility of the modules  $\mathcal{M}(V,\lambda_0)\otimes\bigotimes_{i=1}^m\Omega(\lambda_i,\alpha_i)$ and study the  isomorphisms between these modules, and between these modules and other known modules.

Here follows a brief summary of this  paper.
 In Section $2$, we recall some known results for late use. In
  Section $3$,  the irreducibility of modules $\mathcal{M}(V,\lambda_0)\otimes\bigotimes_{i=1}^m\Omega(\lambda_i,\alpha_i)$ is determined.
 Section 4 is devoted to  giving the necessary and sufficient conditions for two  irreducible $\L$-modules $\mathcal{M}(V,\lambda_0)\otimes\bigotimes_{i=1}^m\Omega(\lambda_i,\alpha_i)$
 and $\mathcal{M}(W,\mu_0)\otimes\bigotimes_{j=1}^n\Omega(\mu_j,\beta_j)$ being isomorphic.    In Section 5,   we compare   irreducible $\L$-modules  constructed in the present paper    with the other known irreducible non-weight $\L$-modules and  show that all these irreducible $\L$-modules $\mathcal{M}(V,\lambda_0)\otimes\bigotimes_{i=1}^m\Omega(\lambda_i,\alpha_i)$ ($V$ being infinite-dimensional and $m\neq1$)   are new.

Throughout this paper, we  denote by $\C,\C^*,\Z$, $\Z_+$, and $\N$  the sets of complex numbers,  nonzero complex numbers, integers ,  nonnegative integers and positive integers, respectively.
 All vector spaces are assumed to be over $\C$.

\section{Preliminaries}

For $\lambda\in\C^*$ and $\alpha\in\C$, the non-weight $\L$-module $\Omega(\lambda,\alpha)=\C[\partial]$  is defined by
$$L_mf(\partial)=\lambda^m(\partial-m\alpha)f(\partial-m)\quad {\rm and}\quad Cf(\partial)=0\quad {\rm  for}\ m\in\Z, f(\partial)\in\C[\partial].$$
 It was proved in \cite{LZ2} that  $\Omega(\lambda,\alpha)$ is irreducible
 if and only if $\alpha \neq0$.

For $r\in\Z_+$,
  denote by $\L_{r}$
the ideal of $\L_{+}=\mathrm{span}_{\C}\{L_i\mid i\geq0\}$  generated by  $L_i$ for all $i>r$.
Let    $\bar \L_{r}$
be the quotient algebra $\L_{+}/ \L_{r}$
 and   $\bar L_i$  the image of $L_i$
in $\bar \L_{r}$. Assume that $V$ is an $\bar \L_{r}$-module.
For any $\lambda\in\C^*$ and  $\alpha\in\C$, define an $\L$-action  on the vector space $\mathcal{M}\big(V,\lambda):=V\otimes \C [\partial]$  as follows (see \cite{LZ}):
 \begin{eqnarray*}
 &&  L_m\big(v\otimes f(\partial)\big)=v\otimes\lambda^m\partial f(\partial-m)+\sum_{i=0}^r(\frac{m^{i+1}}{(i+1)!}\bar L_i)v\otimes \lambda^mf(\partial-m),\\
&& C\big(v\otimes f(\partial)\big)=0\quad {\rm for}\  m\in\Z,v\in V, f(\partial)\in\C[\partial].
 \end{eqnarray*}

\begin{remark}\rm \label{rem2.1} Let $r\in\Z_+$ and let $V$ be  an irreducible $\bar\L_{r}$-module.
\begin{itemize}\lineskip0pt\parskip-1pt
\item[\rm(1)] $V$ must be infinite-dimensional if ${\rm dim}\, V> 1$, since  any   finite-dimensional irreducible module over the solvable Lie algebra $\bar\L_{r}$ is one-dimensional by Lie's Theorem.

\item[\rm(2)] If $r=0,$  then $V$ must be one-dimensional, say,      $V=\C v.$   That is,  $\bar L_iv=0$
for any $i\in\N,$ $\bar L_0v =-\beta v$ for some $\beta\in\C$.  In this case $V$ is denoted by $V_\beta.$  Note that  $\mathcal{M}(V_\beta,\lambda)\cong\Omega(\lambda,\beta)$.
\end{itemize}
\end{remark}

The objects studied in the present paper are the tensor product $\L$-modules of the forms $\mathcal{M}(V,\lambda_0)\otimes\bigotimes_{i=1}^m\Omega(\lambda_i,\alpha_i)$, where  $V$ is an $\bar \L_{r}$-module, $m\in\mathbb N$ and $\lambda_i, \alpha_i\in \C$.
\begin{remark}\rm \label{re2.1} Let $r\in\Z_+$ and $V$ be  an irreducible $\bar\L_{r}$-module.
\begin{itemize}\lineskip0pt\parskip-1pt
\item[\rm(1)]  By Remark \ref{rem2.1}\,(2),  it is clear  that $$\mathcal{M}(V_{\alpha_0},\lambda_0)\otimes\bigotimes_{i=1}^m\Omega(\lambda_i,\alpha_i)\cong
\bigotimes_{i=0}^m\Omega(\lambda_i,\alpha_i)$$
 for any  $\lambda_0,\ldots,\lambda_m\in\C^*$ and $\alpha_0,\ldots,\alpha_m\in\C$. The modules $
\bigotimes_{i=0}^m\Omega(\lambda_i,\alpha_i)$ were studied in \cite{TZ1}.

\item[\rm (2)]  We may assume that $\bar L_r$ is injective on $V$, since otherwise $\bar L_r V=0$ by  {\rm \cite[Lemma 2]{LLZ}}  and    $V$ would  reduce to be an $\bar \L_{r-1}$-module.
\end{itemize}

\end{remark}

 The following result will be used frequently in the sequel.
\begin{prop}{\rm \cite{LZ0}}\label{pro2.3}
Let $P$ be a vector space over $\C$ and $P_1$  a subspace of $P$. Suppose
that $\mu_1,\mu_2,\ldots,\mu_s\in\C^*$ are pairwise distinct, $v_{i,j}\in P$
and $f_{i,j}(t)\in\C[t]$ with $\mathrm{deg}\,f_{i,j}(t)=j$ for $i=1,2,\ldots,s$, $j=0,1,2,\ldots,k.$
If $$\sum_{i=1}^{s}\sum_{j=0}^{k}\mu_i^mf_{i,j}(m)v_{i,j}\in P_1\quad{ \it for}\quad \forall m\in\Z,$$
then $v_{i,j}\in P_1$ for all $i,j\in\Z$.
\end{prop}

\section{Irreducibility}
We  investigate  the irreducibility of the $\L$-modules $\mathcal{M}(V,\lambda_0)\otimes\bigotimes_{i=1}^m\Omega(\lambda_i,\alpha_i)$ in this section.

\begin{lemm}{\rm \cite{TZ1}}\label{lemm3.1}
Let $m\in\N$  and $\lambda_i,\alpha_i\in\C^*$ for $i=0,1,\ldots,m$. Then  $\bigotimes_{i=0}^m\Omega(\lambda_i,\alpha_i)$ is irreducible if and only if
$\lambda_0,\ldots,\lambda_m$ are pairwise distinct.
\end{lemm}

For convenience, we shall  identify
 $\mathcal{M}(V,\lambda_0)\otimes\bigotimes_{i=1}^m\Omega(\lambda_i,\alpha_i)$
with  $V\otimes \C[\partial_0,\ldots,\partial_m]$. And denote $ \underbrace{1\otimes1\otimes\cdots\otimes1}_{m}$ by ${\bf1}_m$ for $m\in\N$.

The following two lemmas are also needed.
\begin{lemm}\label{lemm1}Suppose that $V$ is  an  infinite-dimensional  irreducible  $\bar \L_{r}$-module.
Let $m\in\N, \lambda_0,\lambda_i,\alpha_i\in\C^*$ for $i=1,2,\ldots,m$ with
$\lambda_0,\ldots,\lambda_m$ pairwise distinct and  $0\neq v\in V$.  Then $v\otimes{\bf1}_{m+1}$ generates the module $\mathcal{M}(V,\lambda_0)\otimes\bigotimes_{i=1}^m\Omega(\lambda_i,\alpha_i)$.
\end{lemm}
\begin{proof}
Let $W$ be the submodule of $\mathcal{M}(V,\lambda_0)\otimes\bigotimes_{i=1}^m\Omega(\lambda_i,\alpha_i)$ generated by $v\otimes{\bf1}_{m+1}$.
Note for $k\in\Z$ that
\begin{eqnarray*}
\nonumber L_k(v\otimes{\bf1}_{m+1})
&=& \big(v\otimes\lambda_0^k\partial_0+\sum_{i=0}^{r}(\frac{k^{i+1}}{(i+1)!}\bar L_i)v\otimes \lambda_0^k\big)\otimes\mathbf{1}_{m}+\\
&&\ \ \ v\otimes1\otimes\big(\sum_{i=1}^m1\otimes\cdots\otimes\lambda_i^k(\partial_i-k\alpha_i)\otimes\cdots\otimes1\big)\in W.
\end{eqnarray*}
By Proposition \ref{pro2.3},  $v\otimes\partial_i\in W$ for $i=0,\ldots, m$. While from $L_k(v\otimes\partial_i)\in W$ one can similarly derive $v\otimes \partial_i\partial_j\in W$ for $j=0,\ldots, m$. So inductively one has   $v\otimes\partial_0^{p_0}\cdots\partial_m^{p_m}\in W$ for  $p_i\in\Z_+$ and $i=0,\ldots,m$.
Then it follows from
\begin{eqnarray*}
&& L_k\big(v\otimes (\partial_0+k)^{p_0}\partial_1^{p_1}\cdots\partial_m^{p_m}\big)\in W\\
&=&\big(v\otimes\lambda_0^k\partial_0^{p_0+1}+\sum_{i=0}^{r}(\frac{k^{i+1}}{(i+1)!}\bar L_i)v\otimes \lambda_0^k\partial_0^{p_0}\big)\partial_1^{p_1}\cdots\partial_m^{p_m}+\\
&&\ \ \ v\otimes(\partial_0+k)^{p_0}\otimes\big(\sum_{i=1}^m\partial_1^{p_1}\cdots\lambda_i^k(\partial_i-k\alpha_i)
(\partial_i-k)^{p_i}\cdots\partial_m^{p_m}\big)
\end{eqnarray*} and  Proposition \ref{pro2.3} that $\bar L_iv\otimes \partial_0^{p_0}\partial_1^{p_1}\cdots\partial_m^{p_m}\in W$ for $i=0,1,\ldots,r$. Now by  the irreducibility of $V$,  $\mathcal{M}(V,\lambda_0)\otimes\bigotimes_{i=1}^m\Omega(\lambda_i,\alpha_i)\subseteq W$. Thus, $\mathcal{M}(V,\lambda_0)\otimes\bigotimes_{i=1}^m\Omega(\lambda_i,\alpha_i)$  is generated by $v\otimes {\bf1}_{m+1}$.
\end{proof}

\begin{lemm}\label{lemm2}
Let $r,s\in\Z_+,\lambda,\alpha_1\in\C^*$ and let  $V$ be an infinite-dimensional irreducible $\bar \L_{r}$-module.
Denote $$W_{s, i}=\mathrm{span}\{f(\partial_i)(\partial_0+\partial_1)^n\mid n\in\Z_+,0\leq\mathrm{{deg}}(f(\partial_i))\leq s\}.$$  Then $V\otimes W_{s,i}$  are  submodules    of $\mathcal{M}(V,\lambda)\otimes\Omega(\lambda,\alpha_1)$ for $i=0,1$.
\end{lemm}
\begin{proof}
 For any $0\neq v\in V,n\in\Z_+,k\in\Z$ and $f(\partial_0)\in W_{s,0}$, by noting that $(\partial_0+\partial_1-k)^n\in W_{s,0}$ we have
 \begin{eqnarray*}
 &&L_k\big(v\otimes f(\partial_0)(\partial_0+\partial_1)^n\big)
 =L_k\big(v\otimes\sum_{i=0}^n\binom{n}{i}f(\partial_0)\partial_0^{i}\partial_1^{n-i}\big)\\
&=&v\otimes\sum_{i=0}^n\binom{n}{i}\lambda\partial_0f(\partial_0-k)(\partial_0-k)^{i}\partial_1^{n-i}
\\&&
+\sum_{j=0}^{r}\big(\frac{k^{j+1}}{(j+1)!}\bar L_j\big)v\otimes\sum_{i=0}^n\binom{n}{i}\lambda f(\partial_0-k)(\partial_0-k)^{i}\partial_1^{n-i}
\\&&
+\ v\otimes\sum_{i=0}^n\binom{n}{i}f(\partial_0)\partial_0^{i}\lambda(\partial_1-k\alpha_1)(\partial_1-k)^{n-i}
\\&=&v\otimes\Big(\lambda\partial_0f(\partial_0-k)(\partial_0+\partial_1-k)^n
+f(\partial_0)\lambda(\partial_1-k\alpha_1)(\partial_0+\partial_1-k)^n\Big)
\\&&+\sum_{j=0}^{r}\big(\frac{k^{j+1}}{(j+1)!}\bar L_j\big)v\otimes \lambda f(\partial_0-k) (\partial_0+\partial_1-k)^n
\\&=&v\otimes\lambda\partial_0\big(f(\partial_0-k)-f(\partial_0)\big)(\partial_0+\partial_1-k)^n
\\&&+\ v\otimes\lambda\Big(f(\partial_0)(\partial_0+\partial_1-k)^{n+1}+k(1-\alpha_1)f(\partial_0)
(\partial_0+\partial_1-k)^n\Big)
\\&&+\sum_{j=0}^{r}\big(\frac{k^{j+1}}{(j+1)!}\bar L_j\big)v\otimes \lambda f(\partial_0-k) (\partial_0+\partial_1-k)^n\in V\otimes W_{s,0}.
\end{eqnarray*}  This shows that  $V\otimes W_{s,0}$ is an $\L$-submodule of  $\mathcal{M}(V,\lambda)\otimes\Omega(\lambda,\alpha_1)$. And one can similarly show that
each $V\otimes W_{s,1}$ is also an $\L$-submodule.
\end{proof}

\begin{coro}\label{coro3}
Let $\lambda,\alpha_1\in\C^*$ and let  $V$ be an infinite-dimensional irreducible $\bar \L_{r}$-module.
 Then  $\mathcal{M}(V,\lambda)\otimes\Omega(\lambda,\alpha_1)$
has two series of $\L$-submodules
\begin{eqnarray*}
&&V\otimes W_{0,i}\subset V\otimes W_{1,i}\subset\cdots V\otimes W_{n,i}\subset\cdots
 \end{eqnarray*}such that $V\otimes W_{n,i}/V\otimes W_{n-1,i}\cong \mathcal{M}(V^{(n)},\lambda)$ for $i=0,1$, where $V^{(n)}=V$ is an  $\bar \L_{r}$-module for which the  action
of  $\bar L_0$ is replaced by $\bar L_0-(\alpha_1+n)$ and the actions of other $\bar L_i$ are invariant.
\end{coro}
\begin{proof}
By the similar calculations as in the proof of Lemma \ref{lemm2}, we deduce that
\begin{eqnarray*}
 &&L_k\big(v\otimes \partial_0^n(\partial_0+\partial_1)^m\big)
\\&\equiv&v\otimes\lambda^k\partial_0^n(\partial_0+\partial_1-k(\alpha_1+n))(\partial_0+\partial_1-k)^m
\\&&+\sum_{i=0}^{r}\big(\frac{k^{i+1}}{(i+1)!}\bar L_i\big)v\otimes\lambda^k \partial_0^n (\partial_0+\partial_1-k)^m \quad (\mathrm{mod}\ V\otimes W_{n-1,0})
\end{eqnarray*}
and
\begin{eqnarray*}
&&L_k\big(v\otimes (\partial_0+\partial_1)^m\partial_1^n\big)
\\&\equiv&
v\otimes\lambda^k(\partial_0+\partial_1-k(\alpha_1+n))(\partial_0+\partial_1-k)^m\partial_1^n
\\&&+\sum_{i=0}^{r}\big(\frac{k^{i+1}}{(i+1)!}\bar L_i\big)v\otimes\lambda^k (\partial_0+\partial_1-k)^m \partial_1^n \quad (\mathrm{mod}\ V\otimes W_{n-1,1}).
\end{eqnarray*}
It follows that
$V\otimes  {W}_{n,i}/V\otimes {W}_{n-1,i}\cong \mathcal{M}(V^{(n)},\lambda)$ for $i=0,1$, as desired.
\end{proof}

Now we  are ready to present  the main theorem of this section, which
gives a characterization of   the irreducibility of   $\mathcal{M}(V,\lambda_0)\otimes\bigotimes_{i=1}^m\Omega(\lambda_i,\alpha_i)$.
\begin{theo}\label{th1}
Let $m\in\N,r\in\Z_+$ and $\lambda_0,\lambda_i,\alpha_i\in\C^*$ for $i=1,2,\ldots,m$ with
$\lambda_0,\cdots, \lambda_m$ pairwise distinct. Suppose that   $V$ is  an irreducible $\bar \L_{r}$-module such that the action of $\bar L_r$ on $V$
is injective.  Then the $\L$-module $\mathcal{M}(V,\lambda_0)\otimes\bigotimes_{i=1}^m\Omega(\lambda_i,\alpha_i)$ is  irreducible.
\end{theo}
\begin{proof} Denote  $M=\mathcal{M}(V,\lambda_0)\otimes\bigotimes_{i=1}^m\Omega(\lambda_i,\alpha_i)$.
Consider first   the case in which  $V$ is finite-dimensional.
Then  $V\cong V_{\alpha_0}$ for some  $\alpha_0\in\C^*$ and $M\cong \bigotimes_{i=0}^m\Omega(\lambda_i,\alpha_i)$ by Remarks \ref{rem2.1}\,(2) and \ref{re2.1}\,(1). Thus by   Lemma \ref{lemm3.1}, $M$ is irreducible.

Now assume  that $V$ is infinite-dimensional. For any $u=\sum_{\mathbf{p}\in J} v_{\mathbf{p}}\otimes \partial_0^{p_0}\partial_1^{p_1}\cdots\partial_m^{p_m}\in M$ with  $v_{\bf p}\in V$ nonzero and $J$ a finite subset of $\mathbb Z_+^{m+1},$ define deg$(u)={\rm max}\{{\bf p}\mid {\bf p}\in J\}$ (here we use the  lexicographic order on $\mathbb Z_+^{m+1}$).
  Let  $W$ be a nonzero submodule of $M$ and $w$ a nonzero element  of  $W$ with the minimal degree. We claim $\mathrm{deg}(w)=\mathbf{0}$, i.e., $w=v\otimes {\bf1}_{m+1}$ for some $v\in V$.
Then by Lemma \ref{lemm1},  $W=M$, also proving the irreducibility of  $M$  in this case.

Suppose on the contrary that $\mathrm{deg}(w)>\mathbf{0}$. Write
\begin{equation*}
w=\sum_{\mathbf{p}\in I} v_{\mathbf{p}}\otimes \partial_0^{p_0}\partial_1^{p_1}\cdots\partial_m^{p_m}\in W,
\end{equation*}
where $I$ is a finite subset of $\Z_+^{m+1}$ and each $v_{\mathbf{p}}$ is nonzero in $V$.
Let ${\bf q}$ be  maximal in $I$ and $i^\prime$  minimal such that   $q_{i^\prime}>0$.
Note that for any $k\in\Z_+$,
\begin{eqnarray*}
&& L_k\big(\sum_{\mathbf{p}\in I} v_{\mathbf{p}} \otimes\partial_0^{p_0}\partial_1^{p_1}\cdots\partial_m^{p_m}\big)\\
&=& \sum_{\mathbf{p}\in I}\left(\Big( v_{\mathbf{p}}\otimes\lambda_0^k\partial_0(\partial_0-k)^{p_0}+\sum_{i=0}^{r}\big(\frac{k^{i+1}}{(i+1)!}\bar L_i\big)v_{\mathbf{p}}\otimes \lambda_0^k(\partial_0-k)^{p_0}\Big)\otimes\partial_1^{p_1}\cdots\partial_m^{p_m}\right.\\
   &&\left. \quad\quad \  +~v_{\mathbf{p}}\otimes\partial_0^{p_0}\otimes\sum_{i=1}^m\Big(\partial_1^{p_1}\otimes\cdots\otimes\lambda_i^k(\partial_i-k\alpha_i)
(\partial_i-k)^{p_i}\otimes\cdots\otimes\partial_m^{p_m}\Big)\right)\in W.
\end{eqnarray*}
Applying  Proposition \ref{pro2.3}, we can see that a nonzero element of the  form
$$w^\prime=\left\{\begin{array}{llll}
(-1)^{q_0}\bar L_rv_{{\mathbf{q}}}\otimes \partial_1^{{q}_1}\cdots\partial_m^{{q}_m}+\mathrm{lower}\ \mathrm{terms}&\mbox{if}\
i^\prime=0,\\[4pt]
(-1)^{q_{i^\prime}+1}\alpha_{i^\prime}v_{{\mathbf{q}}}\otimes \partial_{i^\prime}^0\cdots\partial_m^{{q}_m}+\mathrm{lower}\ \mathrm{terms}&\mbox{if}\
1\leq i^\prime\leq m, \alpha_{i^\prime}\neq 0,
\end{array}\right.
$$ lies in $W.$ But ${\rm deg}(w^\prime)<{\rm deg}(w)$, contradicting the minimality of $w$.
\end{proof}

As  an immediate consequence of  Lemma \ref{lemm2} and Theorem \ref{th1} we have:
\begin{coro}
Let $m\in\N,r\in\Z_+,\lambda_0,\lambda_i,\alpha_i\in\C^*$ for $i=1,2,\ldots,m$ and  $V$ be an irreducible $\bar \L_{r}$-module such that  the action of $\bar L_{r}$ on $V$
is injective.  Then the $\L$-module $\mathcal{M}(V,\lambda_0)\otimes\bigotimes_{i=1}^m\Omega(\lambda_i,\alpha_i)$ is  irreducible if and only if
$\lambda_0,\ldots, \lambda_m$ are pairwise distinct.
\end{coro}

\section{Isomorphism classes}
In this section, we will determine   the necessary and sufficient conditions for two classes of  irreducible $\L$-modules $\mathcal{M}(V,\lambda_0)\otimes\bigotimes_{i=1}^m\Omega(\lambda_i,\alpha_i)$ and $\mathcal{M}(W,\mu_0)\otimes\bigotimes_{j=1}^n\Omega(\mu_j,\beta_j)$ to be isomorphic.

\begin{lemm}{\rm \cite{TZ1}}\label{lemm4.1}
Let $m,n\in\N,r,s\in\Z_+$,  $\lambda_i,\alpha_i,\mu_j,\beta_j\in\C^*$ for
 $i=0,\ldots,m,j=0,\ldots,n$ with all
$\lambda_i$ and  all $\mu_j$  respectively pairwise distinct.
Then the $\L$-modules  $\bigotimes_{i=0}^m\Omega(\lambda_i,\alpha_i)
     \cong\bigotimes_{j=0}^n\Omega(\mu_j,\beta_j)$
  if and only  if
 $m=n$ and $(\lambda_i,\alpha_i)=(\mu_{\sigma(i)},\beta_{\sigma(i)})$ for  $i=0,1,\ldots,m,\sigma\in S_{m+1}$
     $\mathrm{(}$the $m+1$-th  symmetric group$\mathrm{)}.$
\end{lemm}

The  main result of this section is as follows:
\begin{theo}\label{th2}
Let $m,n\in\N,r,s\in\Z_+,$ $\lambda_0,\mu_0,\lambda_i,\alpha_i,\mu_j,\beta_j\in\C^*$ for
 $i=1,\ldots,m,j=1,\ldots,n$ with
$\lambda_0,\ldots,\lambda_m$ and  $\mu_0,\ldots,\mu_n$ respectively pairwise distinct. Suppose that  $V$ is an irreducible $\bar \L_{r}$-module on which $\bar L_{r}$ is injective and  $W$ is an irreducible $\bar \L_s$-module on which  $\bar L_s$ is injective.   Then $$\mathcal{M}(V,\lambda_0)\otimes\bigotimes_{i=1}^m\Omega(\lambda_i,\alpha_i)
     \cong\mathcal{M}(W,\mu_0)\otimes\bigotimes_{j=1}^n\Omega(\mu_j,\beta_j)$$
  as $\L$-modules if and only  if one of the following holds
 \begin{itemize}\parskip-3pt
\item[{\rm (a)}] $m=n$,  $V\cong V_{\alpha_0},W\cong V_{\beta_0}$ and $(\lambda_i,\alpha_i)=(\mu_{\sigma(i)},\beta_{\sigma(i)})$ for   $i=0,1,\ldots,m,\sigma\in S_{m+1};$
\item[{\rm (b)}]  $m=n,r=s,$ $V\cong W$
 as $\bar\L_r$-modules and $(\lambda_0,\lambda_i,\alpha_i)=(\mu_0,\mu_{\sigma(i)},\beta_{\sigma(i)})$ for $i=1,\ldots,m$ and some $\sigma\in S_m.$ In this case, $V$ is infinite-dimensional.
\end{itemize}

\end{theo}
\begin{proof}
We only show the ``only if\," part, as the ``if\," part is trivial.
Denote $$\mathcal R=\mathcal{M}(V,\lambda_0)\otimes\bigotimes_{i=1}^m\Omega(\lambda_i,\alpha_i)\quad {\rm and}\quad
\mathcal T=\mathcal{M}(W,\mu_0)\otimes\bigotimes_{j=1}^n\Omega(\mu_j,\beta_j).$$  Let $\phi$ be an isomorphism from $\mathcal R$ to   $\mathcal T$.
For any fixed nonzero  $v\in V$, assume
\begin{equation}\label{ass222}
\phi(v\otimes {\bf1}_{m+1})=\sum\limits_{\mathbf{p}\in I}w_{\mathbf{p}}\otimes\partial_0^{p_0}\partial_1^{p_1}\cdots\partial_n^{p_n},
\end{equation}
where $I$ is a finite subset of $\Z_+^{n+1}$ and all $w_{\mathbf{p}}$ are  nonzero in $W$. Note  from  $\phi(L_k(v\otimes {\bf1}_{m+1}))=L_k\phi(v\otimes {\bf1}_{m+1})$ that we have
\begin{eqnarray}\label{4.42}
&&\lambda_0^k\phi(v\otimes\partial_0\otimes{\bf1}_{m})
+\sum_{i=0}^{r}\lambda_0^k\phi\big((\frac{k^{i+1}}{(i+1)!}\bar L_i)v\otimes {\bf1}_{m+1}\big)\nonumber
\\&&+\ \phi\big(v\otimes1\otimes\sum_{i=1}^m\lambda_i^k(1\otimes\cdots
\otimes(\partial_i-k\alpha_i)\otimes\cdots\otimes1)\big) \nonumber
\\&=&\sum_{\mathbf{p}\in I}\Big(\big( w_{\mathbf{p}}\otimes\mu_0^k{\partial}_0({\partial}_0-k)^{p_0}+\sum_{i=0}^{s}
(\frac{k^{i+1}}{(i+1)!}\bar L_i)w_{\mathbf{p}}\otimes \mu_0^k(\partial_0-k)^{p_0}\big)\otimes\partial_1^{p_1}\cdots\partial_n^{p_n}\nonumber
\\
&&+\,w_{\mathbf{p}}\otimes{\partial}_0^{p_0}\otimes\sum_{j=1}^n\big({\partial}_1^{p_1}
\otimes\cdots\otimes\mu_j^k({\partial}_j-k\beta_j)
({\partial}_j-k)^{p_j}\otimes\cdots\otimes{\partial}_n^{p_n}\big)\Big).
\end{eqnarray}
If $V$ is  finite-dimensional,  then by Remark \ref{rem2.1},  $V\cong V_{\alpha_0}$ for some $\alpha_0\in\C^*$  and the power of $k$ in \eqref{4.42} is less than $1$.
By Proposition \ref{pro2.3},  we observe from \eqref{4.42} that
$m=n$ and   $W\cong V_{\beta_0}$ for some $\beta_0\in\C^*$. Then by  Lemma \ref{lemm4.1}, there exists $\sigma\in S_{m+1}$ such that $(\lambda_i,\alpha_i)=(\mu_{\sigma(i)},\beta_{\sigma(i)})$ for  $i=0,1,\ldots,m$. This is (a).

Consider that $V$ is infinite-dimensional. Then  applying Proposition \ref{pro2.3} in   \eqref{4.42}, we can obtain $m=n$ and
\begin{eqnarray*}
&&\label{4.34rew1} {\rm (I)}: {\rm there\ exists\ } \sigma\in S_m \ {\rm such\ that}\ (\lambda_0,\lambda_i,p_i)=(\mu_0,\mu_{\sigma(i)},0)
\quad{\rm for}\ i=1,\ldots,m,
\\&&\label{4.34rew2}\mathrm{or\ (II)}:{\rm there \ exists\ }  i^\prime> 0\  {\rm such\ that}\ \lambda_0=\mu_{i^\prime},
r=p_{i^\prime}\ {\rm and}\ p_{j}=s=0\ {\rm for}\ j\neq i^\prime.
 \end{eqnarray*}
Expanding
 $\phi\big(L_{l-k}L_{k}(v\otimes\mathbf{1}_{m+1})\big)=L_{l-k}L_{k}\phi(v\otimes \mathbf{1}_{m+1})$,
we have
\begin{eqnarray}\label{aswe4.42}
&&\nonumber
\lambda_0^l\phi\big(\big(v\otimes\partial_0(\partial_0-l+k)
+\sum_{i=0}^{r}\frac{(l-k)^{i+1}}{(i+1)!}\bar L_iv\otimes(\partial_0-l+k)\big)\otimes\mathbf{1}_m\big)
\\&&
+\ \lambda_0^l\phi\big(\big(\sum_{i=0}^{r}\frac{k^{i+1}}{(i+1)!}\bar L_iv\otimes \partial_0
+\sum_{j=0}^{r}\sum_{i=0}^{r}\frac{(l-k)^{j+1}k^{i+1}}{(j+1)!(i+1)!}\bar L_j\bar L_iv\otimes 1\big)\otimes\mathbf{1}_m\big)\nonumber
\\&&\nonumber+\ \lambda_0^k\phi\big((v\otimes\partial_0
+\sum_{j=0}^{r}\frac{k^{j+1}}{(j+1)!}\bar L_j v\otimes 1)\otimes\sum_{i=1}^m\lambda_i^{l-k}(1\cdots(\partial_i-(l-k)\alpha_i)\cdots1)\big) {\color{blue} \otimes ??? }
\\&&\nonumber+\ \lambda_0^{l-k} \phi\big(\big(v\otimes\partial_0+\sum_{j=0}^{r}\frac{(l-k)^{j+1}}{(j+1)!}\bar L_j v\otimes 1\big)\otimes\sum_{i=1}^m\lambda_i^k(1\cdots
(\partial_i-k\alpha_i)\cdots1)\big)
\\&&\nonumber+\ \phi\big(v\otimes1\otimes\sum_{j=1,j\neq i}^m\sum_{i=1}^m\lambda_i^k\lambda_j^{l-k}(1\otimes\cdots(\partial_j-(l-k)\alpha_j)
\cdots(\partial_i-k\alpha_i)\cdots\otimes1)\big)
\\&&+\ \phi\big(v\otimes1\otimes\sum_{i=1}^m\lambda_i^l(1\otimes\cdots\otimes(\partial_i-(l-k)\alpha_i)
(\partial_i-k\alpha_i-l+k)\otimes\cdots\otimes1)\big)
\nonumber
\\&=&\nonumber\sum_{\mathbf{p}\in I}\Big(\big( w_{\mathbf{p}}\otimes\mu_0^l\partial_0(\partial_0-l+k)({\partial}_0-l)^{p_0}
\\&&\nonumber+\ \sum_{i=0}^{s}\frac{(l-k)^{i+1}}{(i+1)!}\bar L_iw_{\mathbf{p}}\otimes\mu_0^l(\partial_0-l+k)({\partial}_0-l)^{p_0}
+\sum_{i=0}^{s}\frac{k^{i+1}}{(i+1)!}\bar L_iw_{\mathbf{p}}\otimes \mu_0^l\partial_0(\partial_0-l)^{p_0}
\\&&\nonumber+\sum_{j=0}^{s}\sum_{i=0}^{s}\frac{(l-k)^{j+1}k^{i+1}}{(j+1)!(i+1)!}\bar L_j\bar L_iw_{\mathbf{p}}\otimes \mu_0^l(\partial_0-l)^{p_0}\big)\otimes\partial_1^{p_1}\cdots\partial_m^{p_m}\nonumber
\\&&\nonumber
+\ \big( w_{\mathbf{p}}\otimes\mu_0^k{\partial}_0({\partial}_0-k)^{p_0}
+\sum_{i=0}^{s}\frac{k^{i+1}}{(i+1)!}\bar L_iw_{\mathbf{p}}\otimes \mu_0^k(\partial_0-k)^{p_0}\big)
\\&&\nonumber
\otimes\ \sum_{i=1}^m\big({\partial}_1^{p_1}
\otimes\cdots\otimes\mu_i^{l-k}({\partial}_i-(l-k)\beta_i)
({\partial}_i-l+k)^{p_i}\otimes\cdots\otimes{\partial}_m^{p_m}\big)
\\&&\nonumber
+\ \big(w_{\mathbf{p}}\otimes\mu_0^{l-k}\partial_0(\partial_0-l+k)^{p_0}
+\sum_{i=0}^{s}\frac{(l-k)^{i+1}}{(i+1)!}\bar L_iw_{\mathbf{p}}\otimes \mu_0^{l-k}(\partial_0-l+k)^{p_0}\big)
\\&&\nonumber
\otimes\ \sum_{j=1}^m\big({\partial}_1^{p_1}
\otimes\cdots\otimes\mu_j^k({\partial}_j-k\beta_j)
({\partial}_j-k)^{p_j}\otimes\cdots\otimes{\partial}_m^{p_m}\big)
\\
&&\nonumber+\ w_{\mathbf{p}}\otimes{\partial}_0^{p_0}\otimes\sum_{j=1,
j\neq i}^m\sum_{i=1}^m\big({\partial}_1^{p_1}\otimes
\cdots\mu_j^{l-k}({\partial}_j-(l-k)\beta_j)
({\partial}_j-l+k)^{p_j}\otimes\cdots
\\&&\nonumber\otimes\ \mu_i^k({\partial}_i-k\beta_i)
({\partial}_i-k)^{p_i}\otimes\cdots{\partial}_m^{p_m}\big)
+w_{\mathbf{p}}\otimes{\partial}_0^{p_0}
\\&&\otimes\ \sum_{i=1}^m\big({\partial}_1^{p_1}\otimes
\cdots\mu_i^l({\partial}_i-(l-k)\beta_i)({\partial}_i-k\beta_i-l+k)
({\partial}_i-l)^{p_i}\cdots\otimes\partial_m^{p_m}\big)\Big).
\end{eqnarray}
Now we apply Proposition  \ref{pro2.3} and consider  the coefficients of
 $k^{2r+2}$ in \eqref{aswe4.42}. In the case (II) one has $\phi(\bar L_r^2v\otimes\mathbf{1}_{m+1})=0$, which is impossible. And in the case (I) one has $r=s$.  It follows from  this and   \eqref{4.42} one has $p_0=0$. That is, we have
$
\phi(v\otimes \mathbf{1}_{m+1})=w_{\mathbf{0}}\otimes\mathbf{1}_{m+1},
$ which allows us to define an injective linear map $\varphi:V\rightarrow W$ such that $\phi(v\otimes \mathbf{1}_{m+1})=\varphi (v)\otimes \mathbf{1}_{m+1}$ for $v\in V$.  Thus,  \eqref{4.42} can be written as
\begin{eqnarray*}
&&\lambda_0^k\phi(v\otimes\partial_0\otimes\mathbf{1}_m)
+\sum_{i=0}^{r}\lambda_0^kk^{i+1}\phi\Big(\big(\frac{1}{(i+1)!}\bar L_i\big)v\otimes \mathbf{1}_{m+1}\Big)\nonumber
\\&&+\ \phi\Big(v\otimes1\otimes\sum_{i=1}^m\lambda_i^k(1\otimes\cdots
\otimes(\partial_i-k\alpha_i)\otimes\cdots\otimes1)\Big)
\\&=&\mu_0^k\varphi (v)\otimes{\partial}_0\otimes\mathbf{1}_m
+\sum_{i=0}^{r}\mu_0^kk^{i+1}\big(\frac{1}{(i+1)!}\bar L_i\big)\varphi (v)\otimes \mathbf{1}_{m+1}\nonumber
\\&&+\ \varphi (v)\otimes1\otimes\sum_{j=1}^m\mu_j^k(1\otimes\cdots\otimes(\partial_j-k\beta_j)\otimes\cdots\otimes1),
\end{eqnarray*}
which together with Proposition \ref{pro2.3}  implies      $$\varphi\big(\bar L_iv\big)=\bar L_i\varphi(v), \ \forall i=0,1,\ldots, r,$$
i.e.,  $V\cong W$ and there exists $\sigma\in S_m$ such that $(\lambda_i,\alpha_i)=(\mu_{\sigma(i)},\beta_{\sigma(i)})$
for  $1\le i\le m$.
That is, we have (b).
\end{proof}

\section{New irreducible modules}
In this section, we compare the tensor product $\L$-modules  $\mathcal{M}(V,\lambda_0)\otimes\bigotimes_{i=1}^m\Omega(\lambda_i,\alpha_i)$
with other known irreducible non-weight Virasoro modules.
By Remark \ref{re2.1}, we only consider the case when $V$ is infinite-dimensional in the following.

Let us first recall  irreducible non-weight $\L$-modules from \cite{H,CH,LGW,LLZ,TZ1,MZ,LZ2}.
 Fix any $\lambda,\alpha\in\C^*$ and $h(t)\in\C[t]$ such that ${\rm \widetilde{deg}}(h(t))=1$. For any $m\in\Z$,  define
$\widehat{h}_m(t)=mh(t)-m(m-1)\alpha\frac{h(t)-h(\alpha)}{t-\alpha}.$ Then $\Phi(\lambda,\alpha,h(t))=\C[\partial,t]$ carries the structure of an irreducible $\L$-module:
$L_mf(\partial,t)=\lambda^m(\partial+\widehat{h}_m(t))f(t,\partial-m)-m\lambda^m(t-m\alpha)\frac{d}{d t}(f(t,\partial-m))$.

Let $V$ be  an irreducible  $\L$-module for which there exists $R_V\in\Z_+$ such that $L_m$ is locally nilpotent on $V$ for all $m\geq R_V$.
It was  shown respectively in \cite{H,LGW,TZ1} that the tensor product  $\L$-modules   $\bigotimes_{i=1}^n\Phi(\lambda_i,\alpha_i,h_i(t))\otimes V$, $\bigotimes_{i=1}^n\Phi(\lambda_i,\alpha_i, h_i(t))\otimes\bigotimes_{j=n+1}^{n+q}\Omega(\mu_j,\beta_j)\otimes V$ and  $\bigotimes_{i=1}^{n}\Omega(\lambda_i,\alpha_i)\otimes V$ are  irreducible if $\lambda_1,\ldots,\lambda_n,\mu_1,\ldots,\mu_m$ are pairwise distinct.


Let $b\in\C$ and $A$ be an irreducible module over the associative algebra $\mathcal{K}=\C[t^{\pm1}, t\frac{d}{dt}].$  The  action of $\L$  on $A_b:=A$ is given by
$$Cv=0, L_nv=(t^{n+1}\frac{d}{dt}+nbt^n)v\quad \mbox{for}\ n\in\Z,v\in A.$$
It was proved in \cite{LZ2} that $A_b$ is an irreducible $\L$-module if and only if one of the following conditions holds: (1) $b\neq 0\ \mbox{or}\ 1$; (2) $b=1$ and $t\frac{d}{dt} A =A$; (3) $b=0$ and $A$ is not isomorphic to the natural $\mathcal{K}$-module $\C[t,t^{-1}]$.

 Let $V$  be an $\bar\L_{r}$-module.
For any $a\in\C$ and $\gamma(t)=\sum_{i}c_it^{i}\in\C[t,t^{-1}]$,
define the  action of $\L$ on $V\otimes \C[t,t^{-1}]$ as follows
\begin{eqnarray*}\label{l6.1}
&L_m(v\otimes t^n)=\Big(a+n+\sum_{i=0}^r\frac{m^{i+1}}{(i+1)!} \bar L_i\Big)v\otimes t^{m+n}+\sum_{i}c_iv\otimes  t^{n+i},\\
&C(v\otimes t^n)=0 \quad {\rm for}\ m,n\in\Z, v\in V.
 \end{eqnarray*}
Then $V\otimes \C[t,t^{-1}]$ carries the structure of an $\L$-module under the  above given actions, which is denoted by ${\mathcal M}(V,\gamma(t))$. Note from \cite{LLZ} that  ${\mathcal M}(V,\gamma(t))$ is a weight $\L$-module if and only if $\gamma(t)\in \C$ and also that
the $\L$-module ${\mathcal M}(V,\gamma(t))$ is irreducible if and only if $V$ is irreducible.

Let  $\C[[t]]$ be the algebra of formal power series.
 Denote
 $e^{mt}=\sum_{i=0}^\infty\frac{(mt)^i}{i!}\in\C[[t]]$ and
 $\H=\mathrm{span}\{L_i\mid i\geq-1\}$.
Assume that $V$ is an $\H$-module.
For any $\mu,\lambda\in\C^*$ and $ \alpha\in\C$, define an $\L$-action  on the vector space $\mathcal{M}\big(V,\mu,\Omega(\lambda,\alpha)\big)=V\otimes \C [\partial]$ as follows
 \begin{eqnarray*}
 &\label{Lm2.1}  L_m\big(v\otimes f(\partial)\big)=v\otimes\lambda^m(\partial-m\alpha)f(\partial-m)+(\mu^me^{mt}-1)\frac{d}{dt}v\otimes \lambda^mf(\partial-m),\\&
\label{C1232.3}  C\big(v\otimes f(\partial)\big)=0\quad {\rm for}\ m\in\Z,v\in V, f(\partial)\in\C[\partial].
 \end{eqnarray*}
 Let $\mathcal C$ denote the set of all nontrivial $\H$-modules $V$ satisfying the following condition:
for any $0\neq v\in V$  there  exists $r_v\in\Z_+$  such  that  $L_{r_v+i}v=0$  for all  $i\geq 1$. Given any  irreducible $\H$-module $V$    in $\mathcal C$,
 it was proved in \cite{CH}  that    the $\L$-module $\mathcal{M}\big(V,\mu,\Omega(\lambda,\alpha)\big)$ is irreducible   if and only if $\alpha\neq0$ and $\mu\neq 1$.

 \begin{prop}\label{5.1}
Let $r\in\N$ and $\lambda_0,\lambda_1,\alpha_1\in\C^*$ with $\lambda_0\neq\lambda_1$. Suppose  that $V$ is an irreducible $\bar \L_{r}$-module such that the action of $\bar L_r$ on $V$
is injective and $V^\prime=\C[L_{-1}]V$($V$ can be seen as an $\L_{+}$-module by setting $L_iV=0$ for $i>r$) is an irreducible $\H$-module in $\mathcal{C}$.
Then the $\L$-modules $\mathcal{M}(V,\lambda_0)\otimes\Omega(\lambda_1,\alpha_1)$ and
$\mathcal{M}\big(V^\prime,\lambda_0\lambda_1^{-1},\Omega(\lambda_1,\alpha_1)\big)$ are isomorphic.
\end{prop}
\begin{proof}
Define a linear  isomorphism
\begin{eqnarray*}
\phi:\quad  \mathcal{M}(V,\lambda_0)\otimes\Omega(\lambda_1,\alpha_1)\quad&\longrightarrow& \mathcal{M}\big(V^\prime,\lambda_0\lambda_1^{-1},\Omega(\lambda_1,\alpha_1)\big)
\\ v\otimes \partial_0^m\otimes\partial_1^n&\longmapsto& \sum_{j=0}^n(-1)^j\binom{n}{j}L_{-1}^{m+j}v\otimes\partial^{n-j}.
\end{eqnarray*}
 Then for any  $k,m,n\in\Z$ and $v\in V$  we have
 \begin{eqnarray}\label{ghbf667}\nonumber
  &&\phi(L_k(v\otimes \partial_0^m\otimes\partial_1^n))
  \\&=&\nonumber \lambda_0^k\phi(v\otimes \partial_0(\partial_0-k)^m\otimes\partial_1^n)
  +\lambda_0^k\phi(\sum_{i=0}^{r}\big(\frac{k^{i+1}}{(i+1)!}\bar L_iv\otimes (\partial_0-k)^m\otimes\partial_1^n)
  \\&&+\lambda_1^k\phi(v\otimes \partial_0^m\otimes(\partial_1-k\alpha_1)(\partial_1-k)^n)
  \end{eqnarray}
  and
  \begin{eqnarray}\label{ghbf668}\nonumber
  &&L_k(\sum_{j=0}^n(-1)^j\binom{n}{j}L_{-1}^{m+j}v\otimes\partial^{n-j})
  \\&=&\nonumber \lambda_1^k(\sum_{j=0}^n(-1)^j\binom{n}{j}L_{-1}^{m+j}v\otimes(\partial-k\alpha_1)(\partial-k)^{n-j})
  \\&&\nonumber+\ \lambda_1^k((\lambda_0\lambda_1^{-1})^ke^{kt}-1)\frac{d}{dt}
  (\sum_{j=0}^n(-1)^j\binom{n}{j}L_{-1}^{m+j})v\otimes (\partial-k)^{n+j}
  \\&=&\nonumber \lambda_1^k\big(L_{-1}^{m}v\otimes(\partial-k\alpha_1)(\partial-L_{-1}-k)^{n}
  -L_{-1}^{m+1}v\otimes(\partial-L_{-1}-k)^{n}\big)
  \\&&+\ \lambda_0^k(L_{-1}-k)^{m}(\sum_{i=0}^{r+1}\frac{k^{i}}{i!}L_{i-1})v\otimes (\partial-L_{-1})^{n}.
  \end{eqnarray} It is worthwhile to point out that in  \eqref{ghbf668}  we have used the formula  $$\sum_{j=0}^n(-1)^j\binom{n}{j}L_{-1}^{m+j}v\otimes\partial^{n-j}
=L_{-1}^{m}v\otimes(\partial-L_{-1})^{n}.$$
Now by
\eqref{ghbf667} and \eqref{ghbf668} we see that    $$\phi(L_k(v\otimes \partial_0^m\otimes\partial_1^n))=L_k(\sum_{j=0}^n(-1)^j\binom{n}{j}L_{-1}^{m+j}v\otimes\partial^{n-j})$$ holds. Thus, $\phi$ is an isomorphism between the two modules.
 \end{proof}

  \begin{remark}\rm
As a co-product of the Proposition \ref{5.1} and Theorem \ref{th1}, we in fact have determined  the irreducibility of the tensor product $\L$-modules
$\mathcal{M}\big(V,\mu,\Omega(\lambda_0,\alpha_0)\big)\otimes\bigotimes_{i=1}^m\Omega(\lambda_i,\alpha_i)$
for  $V\in \mathcal{C},1\neq\mu\in\C,m\geq2$ and $\lambda_i,\alpha_i\in \C^*$.
\end{remark}

\begin{prop}
Suppose that  $V$ is an irreducible $\bar \L_{r}$-module on which  $\bar L_r$ is injective and  $V^\prime$  is an irreducible $\bar \L_{r^\prime}$-module on which    $\bar L_{r^\prime}$ is injective.
Then  $\mathcal{M}(V,\lambda_0)\otimes\bigotimes_{i=1}^m\Omega(\lambda_i,\alpha_i)$ is not isomorphic  to any of the following irreducible $\L$-modules:
\begin{eqnarray*}
&&{\mathcal M}(V^\prime,\gamma(t)),\
   \bigotimes_{j=1}^n\Phi\big(\mu_j,\beta_j,h_j(t)\big)\otimes\bigotimes_{j=n+1}^{n+q}\Omega(\mu_j,\beta_j)\otimes M,\  A_b,\ M,
\end{eqnarray*}
where  $m,n,q\geq1,\lambda_0,\lambda_i,\alpha_i,\mu_j,\beta_j\in\C^*, b\in\C,$ $h_i(t)\in\C[t]$ with ${\rm {deg}}(h_i(t))=1,$ $\lambda_0,\ldots,\lambda_m,$ $\mu_1,\ldots,\mu_{n+q}$ being pairwise distinct$,$ $M$ is an irreducible $\L$-module for which there exists $R_M\in\Z_+$ such that  $L_m$  is locally nilpotent on $M$ for all $m\geq R_M$.
\end{prop}
\begin{proof}Denote $\mathcal{M}(V,\lambda_0)\otimes\bigotimes_{i=1}^m\Omega(\lambda_i,\alpha_i)$ by $\mathcal T$.
Suppose that  $$\phi: \mathcal T\rightarrow {\mathcal M}(V^\prime,\gamma(t))$$  is an  isomorphism of $\L$-modules.
  Take
  any $0\neq v\in V$ and assume that
  \begin{eqnarray*}\label{5.11}
  \phi(v\otimes\mathbf{1}_{m+1})=\sum_{i=1}^n\sum_j b_{i,j} w_i\otimes t^{m+j}
  \end{eqnarray*}
  for $m\in\Z,n\in\N,b_{i,j}\in\C,w_i\in V^\prime$.     Comparing   the
coefficients of $(\lambda_0^{-1}\lambda_1)^kk^{r+2}$ on both sides of
  $$\phi(L_{l-k}L_k(v\otimes\mathbf{1}_{m+1}))=\sum_{i=1}^n\sum_j b_{i,j} L_{l-k}L_k( w_i\otimes t^{m+j}),$$
  gives
   $\phi\big( \bar L_rv\otimes \mathbf{1}_{m+1}\big)=0$, contradicting  $\bar L_rv\otimes \mathbf{1}_{m+1}\neq0$. Thus, $\mathcal T\ncong {\mathcal M}(V^\prime,\gamma(t)).$
And we can similarly conclude that     $\mathcal T\ncong A_b$.

Since the action of $L_m$ for each $m\in\Z$ on  $\mathcal{M}(V,\lambda_0)\otimes\bigotimes_{i=1}^m\Omega(\lambda_i,\alpha_i)$ is not locally
  nilpotent, one has $\mathcal{M}(V,\lambda_0)\otimes\bigotimes_{i=1}^m\Omega(\lambda_i,\alpha_i)\ncong M$.

Denote $\mathcal{P}:=\bigotimes_{j=1}^n\Phi\big(\mu_j,\beta_j,h_j(t)\big)\otimes\bigotimes_{j=n+1}^{n+q}\Omega(\mu_j,\beta_j)\otimes M$.
By the similar argument as  in \cite[Lemma 3]{LLZ} one can verify that $T_{l,k}^{(h)}(\mathcal T)=0$ for $h>2r+2$, where 
$$
T_{l,k}^{(h)}=\sum_{i=0}^h(-1)^{h-i}{h \choose i}L_{l-k-i}L_{k+i}\in\mathcal{U}(\L).$$  Now by \cite[Lemma 5.1]{LGW} we have $\mathcal T\ncong \mathcal{P}.$
This completes the proof.
\end{proof}
Now we can conclude this section with the following corollary.
\begin{coro}
Let $m\in\N\setminus\{1\},r\in\Z_+,\lambda_0,\lambda_i,\alpha_i\in\C^*$ for $i=1,2,\ldots,m$ with
$\lambda_0,\ldots, \lambda_m$ pairwise distinct and  let    $V$ be an irreducible $\bar \L_{r}$-module
 such that the action of $\bar L_{r}$  on  $V$ is injective. Suppose that   $V$ is infinite-dimensional. Then $\mathcal{M}(V,\lambda_0)\otimes\bigotimes_{i=1}^m\Omega(\lambda_i,\alpha_i)$ is not isomorphic to any irreducible $\L$-module in \cite{H,CH,LGW,LLZ,TZ1,MZ,LZ2}.
\end{coro}

\section*{Acknowledgements}
This work was supported by the National Natural Science Foundation of China (No. 11801369,   11871421, 11971350, 12171129), the CSC (No. 202006265002), the Overseas Visiting Scholars Program
of Shanghai Lixin University of Accounting and Finance (No. 2021161),   the Fundamental Research Funds for the Central Universities and the Scientific Research Foundation of Hangzhou Normal University (No. 2019QDL012). The authors would like to thank the referee very much for  many helpful comments and suggestions.

\section*{Data availability}
The data that support the findings of this study are available from the corresponding author upon reasonable request.

\end{document}